\documentclass[11pt,a4paper]{article}

\usepackage{hyperref}
\usepackage{amsmath,caption}
\usepackage{amsfonts}
\usepackage{amssymb}
\usepackage{amsthm}
\usepackage{url,pdfpages,xcolor,framed,color}
\usepackage{a4wide}

\theoremstyle{plain}

\newtheorem{thr}{Theorem}[section]

\newtheorem{lem}[thr]{Lemma}
\newtheorem{prop}[thr]{Proposition}
\newtheorem{conj}[thr]{Conjecture}
\theoremstyle{definition}

\def\A{\mathcal{A}}

\def\S{\mathcal{S}}
\def\F{\mathcal{F}}
\def\C{\mathcal{C}}

\DeclareMathOperator{\VC}{VC-dim}
\DeclareMathOperator{\sh}{sh}

\title{VC dimension and a union theorem for set systems}
\author{Stijn Cambie\footnote{Department of Mathematics, Radboud University Nijmegen, Postbus 9010, 6500 GL Nijmegen, The Netherlands. Email: \href{mailto:S.Cambie@math.ru.nl}{S.Cambie@math.ru.nl}. This author is supported by a Vidi Grant of the Netherlands Organization for Scientific Research (NWO), grant number $639.032.614$.} \and Ant\'onio Gir\~ao\footnote{Department of Pure Mathematics and Mathematical Statistics, University of Cambridge, Wilberforce Road, CB3 0WB Cambridge, UK. Email: \href{mailto:A.Girao@dpmms.cam.ac.uk}{ A.Girao@dpmms.cam.ac.uk}}
\and Ross J. Kang \footnote{Department of Mathematics, Radboud University Nijmegen, Postbus 9010, 6500 GL Nijmegen, The Netherlands. Email: \href{mailto:ross.kang@gmail.com}{ross.kang@gmail.com}.  This author is supported by a Vidi Grant of the Netherlands Organization for Scientific Research (NWO), grant number $639.032.614$.}
}%
\date{}

\begin{document}

	\maketitle
	
	\begin{abstract}
Fix positive integers $k$ and $d$. We show that, as $n\to\infty$, any set system $\A \subset 2^{[n]}$ for which the VC dimension of $\{ \triangle_{i=1}^k S_i \mid S_i \in \A\}$ is at most $d$ has size at most $(2^{d\bmod{k}}+o(1))\binom{n}{\lfloor d/k\rfloor}$. 
Here $\triangle$ denotes the symmetric difference operator.
This is a $k$-fold generalisation of a result of Dvir and Moran, and it settles one of their questions.

A key insight is that, by a compression method, the problem is equivalent to an extremal set theoretic problem on $k$-wise intersection or union that was originally due to Erd\H{o}s and Frankl.

We also give an example of a family $\A \subset 2^{[n]}$ such that the VC dimension of $\A\cap \A$ and of $\A\cup \A$ are both at most $d$, while $\lvert \A \rvert = \Omega(n^d)$. This provides a negative answer to another question of Dvir and Moran.

\end{abstract}
	
\section{Introduction}

Let $\A \subset 2^X$ be a family of subsets of some set $X$.
As usual, we say that $Y \subset X$ is {\em shattered} by $\A$ if the family $\A\cap Y = \{S \cap Y \mid S \in \A\}$ is $2^Y$.
Moreover, we denote by $\sh(\A)$ the set of all subsets of $X$ which are shattered by $\A$. Recall that the {\em VC dimension} of $\A$, denoted by $\VC(\A)$, is the cardinality of the largest $Y \subset X$ in $\sh(\A)$.
We shall assume throughout that $X= [n] := \{1,\dots,n\}$.
Let $\binom{[n]}{\le t }$ denote the family $\{S \subset [n] \mid \lvert S \rvert \le t\}$ and $\binom{n}{\le t }$ its size.

A foundational result regarding the VC dimension of set systems is the Sauer--Shelah--Perles or Sauer--Shelah Lemma. A marginally weaker version of this result was established earlier by Vapnik and \v{C}ervonenkis~\cite{VaCe71}.

\begin{thr}[Sauer~\cite{Sau72}, Perles and Shelah~\cite{She72}]\label{SHPlemma}
Let $d \le n$ be positive integers. For every $\A \subset 2^{[n]}$ with $\VC(\A) \le d$, we have $\lvert \A \rvert \le \binom{n}{ \le d}$.
\end{thr}

\noindent
It is easy to see that this bound is sharp by taking, for example, $\A=\binom{[n]}{\le d}$. 
This bound has fundamental importance and wide applicability, e.g.~in machine learning, model theory, graph theory, and computational geometry.

Let $\star$ be a binary set-operation in $\{ \cap, \cup, \triangle \}$, where $\triangle$ denotes the symmetric difference operator.
We also write $\star \A^k 
= \{S_1 \star \dots \star S_k \mid S_i \in \A, \forall i \in [k]\}$.
Motivated by an application in PAC learnability, Dvir and Moran~\cite{DvMo18+} recently investigated how large $\A$ can be assuming $\A \star \A= \{S\star T \mid S, T \in \A\}$ has bounded VC dimension. 
Using the polynomial method, they proved that $\lvert \A \rvert \le 2\binom{n}{\le\lfloor d/2\rfloor}$ provided $\VC(\A\triangle \A) \le d$. 
They also asked whether an analogous result might hold assuming $\VC(\triangle\A^k) \le d$, particularly for $k=3$~\cite[Qu.~2]{DvMo18+}.

It turns out that this last problem is equivalent to an extremal set theoretic problem about $k$-wise $(n-d)$-union families, as we detail in Section~\ref{sec:equiv}. The provenance of the latter problem is long, predating the notion of VC dimension itself.
For example, through this equivalence, we can observe the following as a consequence of a result of Katona from 1964~\cite{K64}.

\begin{thr}\label{sharp2}
	Let $d<n$ be positive integers with $d\equiv r \pmod 2$ for some $r \in\{0,1\}$.
	For every $\A \subset 2^{[n]}$ with $\VC(\A \triangle \A) \le d$, we have
	$\lvert \A \rvert \le 2^r \binom{n-r}{\le \lfloor d/2 \rfloor}$.
\end{thr}

\noindent
This is a best possible form of the result of Dvir and Moran~\cite{DvMo18+}.

In fact, the question of Dvir and Moran is closely related to a long-standing conjecture of Erd\H{o}s and Frankl~\cite{F79}.
The question is answered by a bound on $k$-wise $(n-d)$-union families that is tight if the ground set $[n]$ is large enough. That bound is an asymptotic form of Erd\H{o}s and Frankl's conjecture and it yields the following theorem.
We provide a proof in Section~\ref{sec:AsProofFrankl}, but remark that it was shown by Frankl~\cite{Fra76} a few years before his conjecture with Erd\H{o}s.

\begin{thr}\label{main}
	Let $k,d$ be positive integers with $d \equiv r \pmod k$ for some $0 \le r \le k-1$.
	There exists $n_0=n_0(d,k)$ such that, for every $n \ge n_0$
 and every $\A \subset 2^{[n]}$ with $\VC(\triangle \A^k) \le d$, we have
$\lvert \A \rvert \le 2^r\binom{n-r}{ \le \lfloor d/k \rfloor }$.
\end{thr}

\noindent
This bound is sharp and 
it completely settles the aforementioned question of Dvir and Moran for every $k$.
Theorem~\ref{main} may be seen as an asymptotic generalisation of Theorems~\ref{SHPlemma} ($k=1$) and~\ref{sharp2} ($k=2$).
Unlike in those two cases, however, the bound in general fails without assuming large enough $n$.

Dvir and Moran noted that the two simple examples $\binom{[n]}{\le d}$ and $\binom{[n]}{\ge n-d}$ preclude analogues of Theorem~\ref{sharp2} if $\star \in \{ \cap, \cup \}$. However, since $S\triangle T = (S\cup T)\setminus (S\cap T)$ for any sets $S,T$, one might wonder if bounding the VC dimension of {\em both} $\A\cap \A$ and $\A\cup \A$ could still lead to a significantly better bound on $\lvert \A \rvert$.
In Section~\ref{Ques1}, we show that this is not the case. Indeed, we construct a family $\A \subset 2^{[n]}$ satisfying $\lvert \A \rvert= \Omega(n^d)$, $\VC(\A \cup \A)\le d$ and $\VC(\A \cap \A) \le d$. This answers another question of Dvir and Moran~\cite[Qu.~1]{DvMo18+} in the negative.

\section{An extremal set theoretic equivalence}\label{sec:equiv}

In this section, we prove that the question of Dvir and Moran~\cite[Qu.~2]{DvMo18+} is equivalent to two older problems in extremal set theory.

For brevity, we define the following parameters, given integers $k,t,d,n>0$ with $t,d< n$:
\begin{itemize}
\item
$m(n,k,t)$ is the size of a largest $\F \subset 2^{[n]}$ that is {\em $k$-wise $t$-intersecting}, i.e.~every member of $\cap \F^k$ has cardinality at least $t$;
\item
$p(n,k,d)$ is the size of a largest $\F \subset 2^{[n]}$ that is {\em $k$-wise $(n-d)$-union}, i.e.~every member of $\cup \F^k$ has cardinality at most $d$; and
\item
$p'(n,k,d)$ is the size of a largest $\F \subset 2^{[n]}$ such that $\VC(\triangle \F^k) \le d$, i.e.~every member of $\sh(\triangle \F^k)$ has cardinality at most $d$.
\end{itemize}
We have chosen our parameter notation to emphasise our problem setting.

Note that easily $p(n,k,d)=m(n,k,n-d)$ always holds.
Due to a connection with the Erd\H{o}s--Ko--Rado Theorem, most previous work on bounding  $m(n,k,t)$ and $p(n,k,d)$ has taken $t=n-d$ fixed.
In contrast, we focus in this paper on $d$ fixed. Put another way, we consider $k$-wise intersecting families with predominant intersections.
For an extensive overview of previous work in the area, we recommend a recent survey by Frankl and Tokushige~\cite{FT16}.

We will use the notion of compression as defined in e.g.~\cite{BR95}. 
For any $i \in [n]$, the {\em $i$-compression} of a family $\A$ is 
 $\C_i(\A)=\{\C_i(A)\mid A \in \A\}$, where
\[
\C_i(S) =
\begin{cases}
S & \text{if $S\in \A$ and $S\setminus \{i\} \in \A$} \\
S \setminus \{i\} & \text{otherwise}
\end{cases}.
\]
After $n$ compressions, we obtain a compressed family, i.e.~a family that is invariant under compressions or, equivalently, under taking subsets. 

With compression we show that $p'(n,k,d)=p(n,k,d)$. First we need the following lemma. 
Note that this lemma is also a consequence of the fact that the trace on a subset $Y$ of a family $\F$ cannot increase by compression, which is proven in~\cite{F83}.

\begin{lem}\label{VCcompressed}
	Let $\A_1, \ldots, \A_k \subset 2^{[n]}$ be families of sets.
	For any $i \in [n]$,
	\[\VC\left( \C_i(\A_1) \triangle \cdots \triangle \C_i(\A_k)\right) \le \VC(\A_1 \triangle \cdots \triangle \A_k).\]
\end{lem}

\begin{proof}
	We prove the stronger statement that \[\sh\left( \C_i(\A_1) \triangle \cdots \triangle \C_i(\A_k)\right) \subset \sh(\A_1 \triangle \cdots \triangle \A_k).\]
	Note that the example $ \A_1=\{\emptyset,[n]\}$ shows that the reverse inclusion is not true in general.
	Let $Y \subset [n]$ be any subset shattered by $\C_i(\A_1) \triangle \cdots \triangle \C_i(\A_k)$.
	If $i \not \in Y,$ then clearly $Y$ is shattered by $\A_1 \triangle \cdots \triangle \A_k$. So assume $i \in Y.$
	Let $R = R'\cup \{i\}$ for some $R'\subset Y$.
	Then $Y \cap (\C_i(S_1) \triangle \cdots \triangle \C_i(S_k))=R$ for some $S_1\in \A_1,\dots,S_k\in \A_k$.
	There is at least one $j\in[k]$ for which $i \in \C_i(S_j)=S_j$ and so both $S_j$ and $S_j \setminus \{i\}$ belong to $\A_j$.
	Note that this implies 
	$\{ Y \cap (S_1 \triangle \cdots  \triangle S_k), Y \cap (S_1 \triangle \cdots \triangle (S_j \setminus \{i\}) \triangle \cdots \triangle S_k)\} = \{R, R \setminus \{i\} \}$.
This proves $Y\in \sh(\A_1 \triangle \cdots \triangle \A_k)$.
 \end{proof}

With this lemma we are ready to prove the equivalence. 

\begin{thr}\label{sharp}
	For every $\A \subset 2^{[n]}$ with $\VC(\triangle \A^k) \le d$, we have
	$\lvert \A \rvert \le p(n,k,d)$.
	Moreover, there are families $\A \subset 2^{[n]}$ with $\VC(\triangle \A^k) \le d$ that meet the bound. That is, $p'(n,k,d)=p(n,k,d)$.
\end{thr}

\begin{proof}
Let $\A \subset 2^{[n]}$ satisfy $\VC(\triangle \A^k) \le d$.
	By Lemma~\ref{VCcompressed}, we may assume that $\A$ is a compressed family and so if $S \in \A$ then $2^S \subset \A.$
	Note that this property also holds for $\triangle \A^k$ and thus $\VC(\triangle \A^k)$ equals the size of a largest union of $k$ elements in $\A$.
	Then $\VC(\triangle \A^k) \le d$ implies that $\A$ 
	 is a $k$-wise $(n-d)$-union family, and so $\lvert \A \rvert \le p(n,k,d)$.
	 
	Taking $\A$ to be any maximum $k$-wise $(n-d)$-union family, we have that $\VC(\triangle \A^k) \le d$ since $\triangle \A^k \subset \binom{[n]}{\le d}$. This implies the bound is sharp.
\end{proof}

	By Theorem~\ref{sharp}, Theorem~\ref{sharp2} follows from an exact bound on $p(n,2,d)$ for every $n$ and $d$ due to Katona~\cite{K64}. 
It is interesting to note that Katona's result can also be shown using compression, as shown by Kleitman~\cite{K66}. 

\section{An asymptotic form of a conjecture of Erd\H{o}s and Frankl}\label{sec:AsProofFrankl} 

In this section, we prove the exact value of $p(n,k,d)$ for all $n$ large enough with respect to $d$ and $k$. 
This is an asymptotic form of a conjecture of Erd\H{o}s and Frankl from the 1970's, cf.~\cite{F79,F91}.
We have reformulated the conjecture to suit our purposes, i.e.~to address~\cite[Qu.~2]{DvMo18+}.

\begin{conj}[Erd\H{o}s and Frankl, cf.~\cite{F79,F91}]\label{F}
	For all integers $n,k,d>0$ with $n \ge d$,
	\[
	p(n,k,d)= \max_{ 0 \le i \le d/k} 2^{d-ki} \binom{n-d+ki}{ \le i}
	\]
\end{conj}

\noindent
As already noted in~\cite{F79}, this conjecture is sharp if true. To see this consider the following families.
Let $\A_{r,i}= \binom{[n-r]}{\le i} \times 2^{[n]\setminus[n-r]}$, for some $0\leq i\leq \lfloor d/k \rfloor$ and $r=d-ki$. Then $\lvert \A_{r,i} \rvert=2^r \binom{n-r}{\le i}$ and the union of any $k$ sets in $\A_{r,i}$ has size at most $ki+r=d$.

Frankl himself gave most attention to the case where $n-d$ is some fixed $t$.
In \cite{F79}, he confirmed Conjecture~\ref{F} when $t = n-d\le C k 2^k$ for some positive constant $C>0$, and also showed the unique extremal examples are isomorphic to some $\A_{r,i}$. 
In~\cite{F91}, he showed the exact ranges of $n,k$ and $t=n-d$ for which $p(n,k,d)=2^d$.
See~\cite{FT16} for further background.

After posting an earlier version of our manuscript, we learned from Frankl that he~\cite{Fra76} had already shown our Theorem~\ref{partialproofFranklConj} below, with a different argument and for a slightly different bound on $n_0$ -- we discuss this at the end of the section. We find it curious that this result of Frankl was not mentioned before in the literature with respect to the conjecture of Erd\H{o}s and Frankl. 
\\

In addition to compression, we also need the notion of shifting as defined in e.g.~\cite{F87}.
For any $i,j \in [n]$, $i<j$, the {\em $(i,j)$-shift} $\S_{ij}(\A)$ of a family $\A$ is $\S_{ij}(\A)=\{\S_{ij}(S)\mid S \in \A\}$, where
\[
\S_{ij}(S) =
\begin{cases}
S \setminus \{j\} \cup \{i\} & \text{if $ i \not \in S, j \in S$ and $S \setminus \{j\} \cup \{i\} \not \in \A$} \\
S & \text{otherwise}
\end{cases}.
\]
After a finite number of shifts, we obtain a shifted family, i.e.~a family that is invariant under shifts.
The following lemma is standard, but for completeness, we give a proof.

\begin{lem}\label{ij_shift}
	If $\A \subset 2^{[n]}$ is a compressed $k$-wise $(n-d)$-union family, then so is $\S_{ij}(\A)$.
\end{lem}

\begin{proof}
	Let $\A \subset 2^{[n]}$ be a compressed $k$-wise $(n-d)$-union family.
	One can check that, if $T=\S_{ij}(S)$ for some $S \in \A$, then $\S_{ij}(2^{S})=2^{T}$. Thus $\S_{ij}(\A)$ is compressed.

	Next, assume for a contradiction that there are $k$ sets $T_1, \ldots, T_k$ in $\S_{ij}(\A)$ whose union $T$ has size $d+1.$
	Here $T_{\ell}= \S_{ij}(S_{\ell})$ where $S_{\ell}\in \A$ for $\ell\in[k]$.
	If $\lvert T \cap \{i,j\} \rvert \le 1$, then it is clear that $\lvert S_1 \cup \cdots \cup S_k \rvert \ge d+1,$ a contradiction with $\A$ being $k$-wise $(n-d)$-union.
	Otherwise, either there is some $T_{\ell}$ with $\{i,j\} \subset T_{\ell}$ or there are sets $T_{\ell}$, $T_q$ with $T_{\ell} \cap \{i,j\}=\{i\}$ and $T_q \cap \{i,j\}=\{j\}$.
	In the former case, $S_{\ell}=T_{\ell} \supset \{i,j\}$. 
	In the latter, by definition both $T_q$ and $T_q \setminus \{j\} \cup \{i\}$ are in $\A$ and so one of the two has union with $S_{\ell}$ equal to $T_q \cup T_{\ell} \supset \{i,j\}$. In either case, we again conclude that $\lvert S_1 \cup \cdots \cup S_k \rvert \ge d+1$, a contradiction.
\end{proof}

\begin{lem}\label{MainLem}
	Let $B \subset [n]$ be a set with $\lvert B \rvert \ge s$ and $\A \subset 2^{[n]}$ a family with $\lvert \A \rvert >2^s \binom{n}{\le u}$. There exists some $A \in \A$ such that $\lvert A \cup B \rvert \ge s+u+1.$
\end{lem}

\begin{proof}
	Take $B' \subset B$ such that $\lvert B' \rvert = s$.
	Let $\A'=\{A \setminus B' \mid A \in \A\}.$
	Then $\lvert \A' \rvert > \binom{n}{\le u}$ and so by definition there is some $A \in \A$ such that $\lvert A \setminus B' \rvert >u$, and thus $\lvert A \cup B \rvert \ge \lvert A \cup B' \rvert >s+u$.
\end{proof}

We are now prepared to prove Conjecture~\ref{F} for $n$ large enough compared with $d$ and $k$.
We first prove it for $d \equiv 0 \pmod k$, which is the base case in the general proof.

\begin{prop}\label{gen01}
	Let $k,d$ be positive integers with $d \equiv 0 \pmod k$.
	There exists $n_0=n_0(d,k)$ such that $p(n,d,k)=\binom{n}{ \le d/k}$ for every $n \ge n_0$.
	Moreover, the only $k$-wise $(n-d)$-union family of size $\binom{n}{ \le d/k}$ equals $\binom{[n]}{ \le d/k}$.
\end{prop}

\begin{proof}
	Let $t=d/k$ and $\A$ be a family with $\lvert \A \rvert >\binom{n}{ \le t}$.
	Choose $n_0$ such that $\binom{n}{\le t}-2^{(k-1)t+1}\binom{n }{\le t-1} > 0$ for every $n \ge n_0$. Such a choice exists because we have a polynomial in $n$ whose leading coefficient is strictly positive.
	We prove by induction on $i\in[k]$ that there exist sets $A_1, \ldots, A_i \in \A$ such that $\lvert A_1 \cup \cdots \cup A_i \rvert \ge it+1$.
	If $i=1$, then the statement is trivial.
	Assume it holds for some $i \in [k-1]$.
	Then by the choice of $n_0$ we can apply Lemma~\ref{MainLem} to $\A$ with $B=A_1 \cup \cdots \cup A_i$, $s=it+1$ and $u=t-1$ for the inductive step.
	This proves that $\A$ is not $k$-wise $(n-d)$-union.

Note that this induction argument also proves that if $\A$ satisfies $\lvert \A \rvert =\binom{n}{ \le t}$ then $ \A = \binom{[n]}{ \le t}$, thus proving uniqueness of the extremal example.
\end{proof}

\begin{thr}[Frankl~\cite{Fra76}]\label{partialproofFranklConj}
	Let $k,d$ be positive integers with $d \equiv r \pmod k$ for some $0 \le r \le k-1$.
	There exists $n_0=n_0(d,k)$ such that $p(n,k,d)=2^r\binom{n-r}{ \le \lfloor d/k \rfloor }$ for every $n \ge n_0$.
	Moreover, the only $k$-wise $(n-d)$-union family of size $\binom{n}{ \le d/k}$ equals $\A_{ r, \lfloor d/k \rfloor }$ up to relabelling.
\end{thr}

\begin{proof}
	The proof is by induction on $r$, $0\le r\le k-1$.
	The base case $r=0$ is Proposition~\ref{gen01}.
	So assume $r \ge 1$.
	Fix any $k>r$ and $d \equiv r \pmod k$. Write $t=\lfloor d/k \rfloor$ which equals $(d-r)/k$.
	Since $d-1 \equiv r-1 \pmod k$, by induction there exists $n_0(d-1,k)$ such that $p(n-1,k,d-1)= 2^{r-1}\binom{n-r}{ \le t }$ for every $n\ge n_0(d-1,k)+1$.
	Choose $n_0 \ge n_0(d-1,k)+1$ large enough such that 
	$2^r \binom{n-r}{ \le t}-\left(2^{d+1} \binom{n}{ \le t-1} + \binom{n}{t}\right) > 0$ holds for every $n \ge n_0$. Such a choice exists because we have a polynomial in $n$ whose leading coefficient is strictly positive, as $r\ge 1$.
	
	For $n \ge n_0$, take a maximum family $\A \subset 2^{[n]}$ which is $k$-wise $(n-d)$-union. So $\lvert \A \rvert =  p(n,k,d)$.
	We may assume $\A$ is compressed and shifted by Lemma~\ref{ij_shift} since any maximal $k$-wise $(n-d)$-union family is necessarily invariant under taking subsets. .
	
	Let $\A_1=\{S \in \A \mid 1 \in S\}$ and $\A_{\overline1}=\A \setminus \A_1.$
	If $\lvert \A_{\overline1} \rvert \le p(n-1,k,d-1)$, then the result follows since, by induction, we have $\lvert \A_{1} \rvert \le p(n-1,k,d-1)$ and $\lvert \A \rvert= \lvert \A_{\overline1} \rvert+\lvert \A_1 \rvert \leq 2\cdot p(n-1,k,d-1) = 2^r\binom{n-r}{ \le t }$. Otherwise, by definition $\A_{\overline1}$ contains $k$ sets $S_1, \ldots, S_k$ whose union has size $d$.
	First order the sets in nonincreasing size: $\lvert S_1 \rvert \ge \cdots \ge \lvert S_k \rvert$.
	As $\A$ is compressed and shifted, we may assume that $S_1, \ldots, S_k$ are disjoint and their union is $[d+1]\setminus\{1\}$.
	Then $\lvert S_k \rvert \le d/k$, and so $\lvert S_k \rvert \le t$.
	There cannot be a set $S'_k$ in $\A$ which contains $1$ and $t$ elements of $[n]\setminus[d+1]$, or else $\lvert S_1 \cup \cdots \cup S_{k-1} \cup S'_k \rvert \ge d+1$, contradicting that $\A$ is $k$-wise $(n-d)$-union.
	As $\A$ is shifted, it contains at most 
	$2^{d+1} \binom{n-(d+1)}{\le t-1}+\binom{n-(d+1)}{t} < 2^r \binom{n-r}{\le t}$ sets.
	This completes the inductive step.
	
	Note that equality occurs if and only if $\lvert \A_1 \rvert =\lvert \A_{\overline1} \rvert = p(n-1,k,d-1)$ and so uniqueness up to relabelling of the maximal $k$-wise $(n-d)$-union families also follows by induction.
\end{proof}

Theorems~\ref{sharp} and~\ref{partialproofFranklConj} together imply Theorem~\ref{main}.
From our proof we deduce that $n_0(d,k)$ in Theorems~\ref{main} and~\ref{partialproofFranklConj} can be taken to be of order $d2^d/k$. Note that  
it cannot be of order smaller than $d2^k/k$, by the examples stated just after Conjecture~\ref{F}.
We remark that Frankl~\cite{Fra76} originally employed a different type of induction for a more general result proving an upper bound on $n_0(d,k)$ of order $d^32^d/k^2$.

	\section{A counterexample to a question of Dvir and Moran}\label{Ques1}

Dvir and Moran~\cite[Qu.~1]{DvMo18+} asked if it could be true that a set system $\A \subset 2^{[n]}$ satisfies $\lvert \A \rvert \le n^{d/2+O(1)}$ whenever $\VC(\A \cup \A)\le d$ and $\VC(\A \cap \A) \le d$. We show that this is not the case.
	
\begin{prop}\label{better_example}
For each $d \le n$, there exists $A \subset 2^{[n]}$ satisfying $\VC(\A \cap \A) \le d,$ $\VC( \A \cup \A) \le d$ and $\lvert \A \rvert > (n/d)^d$. 
\end{prop}
	
\begin{proof}
Let $\A \subset 2^{[n]}$ be the family of subsets of $[n]$ that satisfies the property ``monotonicity modulo $d$'', i.e.~let $S\subset [n]$ belong to $\A$ if  $i-d\in S$ for any $i\in S$ with $i>d$.
We note that every set $S \in \A$ can be uniquely represented by $d$ integers $i_1,\dots, i_d$ with $0 \le i_k \le \lfloor ({n-k})/{d}+1 \rfloor$: write $S= \cup_{k=1}^d S_k$, where $S_k=\{k,k+d, \ldots, k+(i_k-1) d\}$ for $k \in [d]$.
We now verify that this family $A$ satisfies the required properties.
Note that $\A \cup \A = \A \cap \A = \A$, since the property ``monotonicity modulo $d$'' is preserved by intersection or union.
\begin{itemize}
\item We have $\VC(\A)=d$. 
First note that $\VC(A) \ge d$ since $2^{[d]} \subset \A$.
Next we show the reverse inequality.
Let $Y \subset [n]$ be a subset of size at least $d+1$.
By the pigeonhole principle, $Y$ contains two elements $y_1, y_2$ such that $y_1 \equiv y_2 \pmod d,$ where without loss of generality we may assume $y_2>y_1$.
Due to the property ``monotonicity modulo $d$'', every set $S \in A$ containing $y_2$ contains $y_1$ as well.
Thus there is no $S \in \A$ such that $\{y_2\}=S \cap Y$, and so $Y\notin \sh(A)$.
\item The family $\A$ has size $\lvert \A \rvert = \prod_{k=1}^d \lfloor (n-k)/d +2 \rfloor > ( n/d )^d$. \qedhere	
	\end{itemize}
	\end{proof}
	
	Obviously, since $\A \subset \A \cup \A$, we know by  Theorem~\ref{SHPlemma} that $\lvert \A \rvert \le \binom{n}{ \le d} \le (1+d)(en/d)^d$, so the construction is best possible up to a factor depending on $d$. 

When $d=1$, up to relabelling Proposition~\ref{better_example} gives the unique extremal families: complete chains, i.e.~families of $n+1$ subsets of $[n]$, ordered by inclusion.
The upper bound $n+1$ is a consequence of Theorem~\ref{SHPlemma} and $\A \subset \A \cap \A$. One can check uniqueness by noting that another candidate would contain two equal-sized subsets $S_1, S_2 \subset [n]$ and then performing a small case distinction.

On the other hand, we observe that Proposition~\ref{better_example} is not tight in general.
	For example take $n=d+1$ with $d\ge 3$, then the maximum size of a family $\A \subset 2^{[n]}$ satisfying $\VC(\A \cap \A) \le d$ and $\VC( \A \cup \A) \le d$ equals $2^n-2>3 \cdot 2^{d-1}=\prod_{k=1}^d \lfloor (n-k)/d +2 \rfloor$.
	Indeed, it is easy to see that for every family $\A \subset 2^{[n]}$ of size $2^n-1$, either $\A \cap \A$ or $\A \cup \A$ equals $2^{[n]}$. Furthermore, the family $\A=2^{[n]}\setminus \{ [1], [n]\}$ has size $2^{n}-2$ and $\A \cap \A=2^{[n]}\setminus \{  [n]\} $ and $\A \cup \A=2^{[n]}\setminus \{ [1] \}$. Clearly, both $\A \cap \A$ and  $\A \cup \A$ have VC dimension $d=n-1.$

\subsection*{Acknowledgement}

We thank Peter Frankl for informing us of his early work in~\cite{Fra76,F83}.

\bibliographystyle{abbrv}
\bibliography{VCunion}

\begin{thebibliography}{10}

\bibitem{BR95}
B.~Bollob\'as and A.~J. Radcliffe.
\newblock Defect {S}auer results.
\newblock {\em J. Combin. Theory Ser. A}, 72(2):189--208, 1995.

\bibitem{DvMo18+}
Z.~{Dvir} and S.~{Moran}.
\newblock {A Sauer-Shelah-Perles Lemma for Sumsets}.
\newblock {\em ArXiv e-prints}, June 2018.

\bibitem{Fra76}
P.~Frankl.
\newblock Families of finite sets satisfying union restrictions.
\newblock {\em Studia Sci. Math. Hungar.}, 11(1-2):1--6 (1978), 1976.

\bibitem{F79}
P.~Frankl.
\newblock Families of finite sets satisfying a union condition.
\newblock {\em Discrete Math.}, 26(2):111--118, 1979.

\bibitem{F83}
P.~Frankl.
\newblock On the trace of finite sets.
\newblock {\em J. Combin. Theory Ser. A}, 34(1):41--45, 1983.

\bibitem{F87}
P.~Frankl.
\newblock The shifting technique in extremal set theory.
\newblock In {\em Surveys in combinatorics 1987 ({N}ew {C}ross, 1987)}, volume
  123 of {\em London Math. Soc. Lecture Note Ser.}, pages 81--110. Cambridge
  Univ. Press, Cambridge, 1987.

\bibitem{F91}
P.~Frankl.
\newblock Multiply-intersecting families.
\newblock {\em J. Combin. Theory Ser. B}, 53(2):195--234, 1991.

\bibitem{FT16}
P.~Frankl and N.~Tokushige.
\newblock Invitation to intersection problems for finite sets.
\newblock {\em J. Combin. Theory Ser. A}, 144:157--211, 2016.

\bibitem{K64}
G.~Katona.
\newblock Intersection theorems for systems of finite sets.
\newblock {\em Acta Math. Acad. Sci. Hungar}, 15:329--337, 1964.

\bibitem{K66}
D.~J. Kleitman.
\newblock On a combinatorial conjecture of {E}rd{\H{o}}s.
\newblock {\em J. Combin. Theory}, 1:209--214, 1966.

\bibitem{Sau72}
N.~Sauer.
\newblock On the density of families of sets.
\newblock {\em J. Combin. Theory Ser. A}, 13:145--147, 1972.

\bibitem{She72}
S.~Shelah.
\newblock A combinatorial problem; stability and order for models and theories
  in infinitary languages.
\newblock {\em Pacific J. Math.}, 41:247--261, 1972.

\bibitem{VaCe71}
V.~N. Vapnik and A.~J. \v{C}ervonenkis.
\newblock The uniform convergence of frequencies of the appearance of events to
  their probabilities.
\newblock {\em Teor. Verojatnost. i Primenen.}, 16:264--279, 1971.

\end{thebibliography}

\end{document}